\documentclass[10pt,reqno]{amsart}
\usepackage{amsmath}
\usepackage{amsfonts}
\usepackage{amssymb}
\usepackage{amsthm}
\usepackage{amscd}

\newcommand{\Section}[1]{\section{#1} \setcounter{equation}{0}}

\newtheorem{thm}{Theorem}
\newtheorem{lem}{Lemma}

\theoremstyle{remark}

\theoremstyle{definition}

\begin{document}
\large
\title{Mean Value Theorems for Binary Egyptian Fractions}
\author{Jingjing Huang and Robert C. Vaughan}
\address{JH: Department of Mathematics, McAllister Building,
Pennsylvania State University, University Park, PA 16802-6401,
U.S.A.}
\email{huang@math.psu.edu}
\address{RCV: Department of Mathematics, McAllister Building,
Pennsylvania State University, University Park, PA 16802-6401,
U.S.A.} 
\email{rvaughan@math.psu.edu}
\thanks{Research supported by NSA grant number H98230-09-1-0015.\\}
\begin{abstract}
In this paper, we establish two mean value theorems for the number of solutions of the Diophantine equation 
$\frac{a}{n}=\frac{1}{x}+\frac{1}{y}$, in the case when $a$ is fixed and $n$ varies and in the case when both $a$ and $n$ vary. 
\end{abstract}
\maketitle
\Section{Introduction} 
\noindent The solubility of the diophantine equation
\begin{equation}
\label{e:eq1}
\frac{a}{n}=\frac{1}{x_1}+\frac{1}{x_2}+\cdots+\frac{1}{x_k},
\end{equation}
in positive integers $x_1, x_2, \ldots, x_k$ has a long history. See, for example, Guy \cite{Guy} 
for a detailed survey on this topic and a more extensive bibliography.  When $k\ge3$ it is still an open question as to whether the equation is always soluble provided that $n>n_0(a,k)$.  When $k=3$ the strongest result in this direction is Vaughan \cite{Vau}, \cite{Vau1} (see also Zun \cite{Shen}, and Viola \cite{Vio} for a related equation).  In this memoir we are concerned with the case $k=2$.  In that case it is known that for any given $a>2$ there are infinitely many $n$ for which the equation is insoluble.  For example, the criterion enunciated in the first paragraph of \S3 shows that no $n$ with all its prime factors $p$ of the form $p\equiv 1\pmod a$ has such a representation.  However the number
\begin{equation}
\label{e:eq2}
R(n;a)={\rm{card}}\left\{(x, y)\in \mathbb{N}^2:\frac{a}{n}=\frac{1}{x}+
\frac{1}{y}\right\}
\end{equation}
of representations has an interesting and complicated multiplicative structure and can be studied in a number of ways.  Here we consider various averages
\begin{equation}
\label{e:eq3}
S(N;a)=\sum_{\substack{n\leq N\\(n,a)=1}}R(n;a),
\end{equation}
$$T(N;a)=\sum_{n\leq N}R(n;a)$$
and
$$U(N)=\sum_{a}S(N;a).$$
Croot, {\it et al} \cite{Cro} have shown that 
$$U(N)=\frac{1}{4} CN (\log N)^3+\O\left(\frac{N(\log N)^3}{\log\log N}\right),$$
and in Theorem 2 below we obtain a significant strengthening.  However, in the main result of this paper, Theorem 1, below, we show that it is possible to obtain a strong asymptotic formula without the necessity of averaging over $a$.

\begin{thm}
\label{t:thm1}
$$S(N;a)=\frac{3}{\pi^2a}\Bigg(
\prod_{p|a}\frac{p-1}{p+1}
\Bigg)N\big((\log N)^2+c_1(a)\log{N}+c_0(a)\big)+\Delta(N;a)$$
where 
$$c_1(a)=6\gamma-4\frac{\zeta'(2)}{\zeta(2)} -2 +\sum_{p|a}\frac{6p+2}{(p-1)^2}\log p$$
and 
$$c_0(a)=- (\log a)^2+(\log a)\sum_{p|a}\frac{\log p}{p-1} +O( a\phi(a)^{-1}\log a),$$
and
$$\Delta(N;a)\ll N^{\frac12}(\log(N))^5\frac{a}{\phi(a)}\prod_{p|a}\left(
1-p^{-1/2}
\right)^{-1}$$
uniformly for $N\ge 4$ and $a\in\mathbb N$.
\end{thm}
\bigskip 

Since 
$$T(N;a)=\sum_{d|a} S\left(
\frac{N}d;\frac{a}d
\right)$$
it is a straightforward exercise to obtain the corresponding asymptotic expansion for $T$.\par

The main novelty in this paper is the employment, for the first time in this area, of complex analytic techniques from multiplicative number theory.  In view of this the referee has speculated on the utility of assuming the Generalised Riemann Hypothesis (GRH) in possibly improving the error term here significantly.  This is unlikely with the proof in its present form, since the the main theoretical input from Dirichlet $L$--functions is {\it via} Lemma \ref{l:lem5} below and the bounds there are at least as strong as can be established on GRH apart possibly from the power of the logarithm.  However, in view of the aforementioned criterion in \S2, the underlying problem has some affinity with the generalised divisor problem in the case of $d_3(n)$ and it is conceivable that, by pursuing methods related to that problem, an error bound of the form 
$$O(N^{\theta})$$ 
can be obtained with 
$$\frac13 <\theta <\frac12.$$
\begin{thm} \label{t:thm2}
We have $$U(N)=\frac{1}{4} CN (\log N)^3+O\big(N(\log N)^2\big),$$
where $\displaystyle C=\prod_p\left(1-3p^{-2}+2p^{-3}\right)$.
\end{thm}

The referee has drawn our attention to the Zentralblatt review of \cite{Cro} where the reviewer adumbrates a proof of a result somewhat weaker than Theorem \ref{t:thm2}. 

This paper is organized as follows. In \S2, we state several lemmas which are needed in the proof of Theorem \ref{t:thm1}. In \S3, we present an analytic proof of Theorem \ref{t:thm1} based on Dirichlet L-functions. And in \S4, an essentially elementary proof of Theorem \ref{t:thm2} is given. Finally, in \S5, we list some open questions in this area.

\Section{Preliminary Lemmas}

We state several lemmas before embarking on the proof of Theorem \ref{t:thm1}.  The content of Lemma \ref{l:lem1} can be found, for example, in Corollary 1.17 and Theorem 6.7 of Montgomery \& Vaughan \cite{MV}, and Lemma \ref{l:lem2} can be deduced from Theorem 4.15 of Titchmarsh \cite{Tit} with $x=y=(|t|/2\pi)^{1/2}$.
\newtheorem{lemma}[thm]{Lemma}
\begin{lem}
\label{l:lem1}
When $\sigma\geq 1$ and $|t|\geq2$, we have $$\frac{1}{\log|t|}\ll \zeta(\sigma+it)\ll \log|t|.$$
\end{lem}

\begin{lem}
\label{l:lem2}
When $0\leq\sigma\leq1$ and $|t|\geq2$, we have
$$\zeta(\sigma+it)\ll |t|^\frac{1-\sigma}{2}\log(|t|).$$
\end{lem}

\begin{lem}
\label{l:lem3}
Let $\chi$ be a non-principle character modulo $a$ and $s=\sigma+it$ and assume that $t\in\mathbb R$.  Then
$$L(s,\chi)\ll \log(a(2+|t|)),\text{ when }\sigma\geq 1$$
and
$$L(s,\chi)\ll \left(a|t|\right)^{\frac{1-\sigma}{2}+\varepsilon},\text{ when }\frac{1}{2}\leq\sigma\leq1.$$
\end{lem}

\begin{proof}
The first part follows from Lemma 10.15 of Montgomery \& Vaughan \cite{MV}.  Now suppose that $\chi$ is primitive.  Then by Corollary 10.10 of Montgomery \& Vaughan \cite{MV}, 
$$L(s,\chi)\ll \left(a|t|\right)^{\frac12-\sigma}\log (a(2+|t|))$$
when $\sigma\le 0$.  Then by the convexity principle for Dirichlet series, for example as described in Titchmarsh \cite{Tit} (cf. exercise 10.1.19 of Montgomery \& Vaughan \cite{MV}),
$$L(s,\chi)\ll \left(a|t|\right)^{\frac{1-\sigma}{2}+\varepsilon}$$
when $0\le \sigma\le 1$.  The proof is completed by observing that if $\frac12\le\sigma\le 1$ and $\chi$ modulo 
$a$ is induced by the primitive character $\chi^*$ with conductor $q$, then 
$$L(s,\chi) = 
L(s,\chi^*)\prod_{\substack{p|a\\ p\nmid q}}(1-\chi^*(p)p^{-s}) \ll |L(s,\chi^*)|2^{\omega(a)}.$$
\end{proof}

\begin{lem}
\label{l:lem4}
Let $T\geq 2$, then we have
$$\int_{-T}^{T}|\zeta(\textstyle{\frac{1}{2}}+it)|^4dt\sim \frac{1}{\pi^2}T\log^4T$$
and
$$\sideset{}{^*}\sum_{\substack{\chi\\ \mod{a}}}\int_{-T}^T|L(\textstyle{\frac{1}{2}}+it,\chi)|^4d t\ll \phi(a) T (\log(aT))^4,$$
where $\sideset{}{^*}\sum$ indicates that the sum is over the primitive characters modulo $a$. 
\end{lem}

The first formula here is due to Ingham \cite{Ing} and the second is Theorem 10.1 of Montgomery \cite{Mon}.

\begin{lem}
\label{l:lem5}
Let $T\geq 2$, then
$$\sum_{\substack{\chi\\ \mod{a}}}\int_{-T}^T|L(\textstyle{\frac{1}{2}}+it,\chi)|^4d t\ll a T (\log(aT))^4,$$
\end{lem}
\begin{proof}
Suppose that the character $\chi$ modulo 
$a$ is induced by the primitive character $\chi^*$ with conductor $q$.   Then the $L$--function in the integrand 
in modulus is
$$|L(\textstyle{\frac{1}{2}}+it,\chi^*)\prod_{p|a,p\nmid q}(1-\chi^*(p)p^{-1/2-it})|\le |L(\textstyle{\frac{1}{2}}+it,\chi^*)|\prod_{p|a/q}(1+p^{-1/2}).$$
Hence by the previous lemma
$$\sum_{\substack{\chi\\ \mod{a}}}\int_{-T}^T|L(\textstyle{\frac{1}{2}}+it,\chi)|^4d t\ll  T (\log(aT))^4\sum_{q|a}\phi(q)\prod_{p|a/q}(1+p^{-1/2})^4.$$
The sum here is
\begin{align*}
\sum_{q|a}\phi(q)\prod_{p|a/q}&\left(
1+p^{-\frac12}
\right)^4\\
&=
\prod_{p^k\|a}\left(
\left(
1+p^{-\frac12}
\right)^4+\sum_{h=1}^{k-1} \phi(p^h)\left(
1+p^{-\frac12}
\right)^4+\phi(p^k)
\right)\\
&=a\prod_{p|a}\left(
1+p^{-1}\big((1+p^{-1/2})^4-1\big)
\right)\ll a.
\end{align*}
\end{proof}
\vfill
\break
\Section{Proof of Theorem 1}
\noindent 
Without loss of generality, we can assume $a\le 2N$, since $R(n;a)=0$ whenever $a>2n$. 
Now we rewrite the equation $\frac{a}{n}=\frac{1}{x}+\frac{1}{y}$ in the form
$$(ax-n)(ay-n)=n^2.$$
After the change of variables $u=ax-n$ and $v=ay-n$, it follows that $R(n;a)$ is the number of ordered pairs of natural numbers $u$, $v$ such that $uv=n^2$ and $u\equiv v\equiv -n\pmod{a}$.\par
Under the assumption that $(n,a)=1$, $R(n;a)$  can be further reduced to counting the number of divisors $u$ of $n^2$ with $u\equiv -n\pmod{a}$.  Now the residue class $u\equiv -n \pmod{a}$ is readily isolated {\it via} the orthogonality of the Dirichlet characters $\chi$ modulo $a$.  Thus we have
\begin{align*}
S(N;a)&=\sum_{\substack{n\leq N\\(n,a)=1}}R(n;a)\\
        &=\sum_{\substack{n\leq N\\(n,a)=1}}\frac{1}{\phi(a)} \sum_{\substack{
\chi\\
\mod{a}}} \Bar{\chi}(-n)\sum_{u|n^2}\chi(u)\\
        &=\frac{1}{\phi(a)}\sum_{\substack{\chi\\
\mod{a}}} \Bar{\chi}(-1)\sum_{n\leq N}\Bar{\chi}(n)\sum_{u|n^2}\chi(u),
\end{align*}
where the condition $(n,a)=1$ is taken care of by the character $\bar{\chi}(n)$.\par
Let 
\begin{equation}
a_n(\chi)=\displaystyle\Bar{\chi}(n)\sum_{u|n^2}\chi(u).
\label{e:e31}
\end{equation} 
Then we have
$$S(N;a)=\frac{1}{\phi(a)}\sum_{\substack{\chi\\ \mod{a}}}\Bar{\chi}(-1)\sum_{n\leq N}a_n(\chi).$$
We analyze this expression through the properties of the Dirichlet series $$f_\chi(s)=\displaystyle\sum_{n=1}^{\infty}\frac{a_n(\chi)}{n^s}.$$ 
\vfill
\break
The condition $u|n^2$ can be rewritten uniquely as $u=n_1n_{2}^2$ and $n=n_1n_2n_3$ with $n_1$ square-free.  Hence, for $\sigma>1$ we have
\begin{align*}
f_\chi(s)&=\sum_{n=1}^{\infty}\frac{\Bar{\chi}(n)}{n^s}\sum_{u|n^2}\chi(u)\\
    &=\sum_{n_1,n_2,n_3=1}^{\infty}\mu(n_1)^2\frac{\Bar{\chi}(n_1n_2n_3)\chi(n_1n_2^2)}{n_{1}^s n_{2}^s n_{3}^s}\\
    &=\sum_{n_1=1}^{\infty}\frac{\mu(n_1)^2\chi_0(n_1)}{n_{1}^{s}}\sum_{n_2=1}^{\infty}\frac{{\chi}(n_2)}{n_{2}^{s}}
    \sum_{n_3=1}^{\infty}\frac{\Bar{\chi}(n_3)}{n_{3}^{s}}
\end{align*}
and so
\begin{equation}
f_\chi(s)=\frac{L(s,\chi_0)}{L(2s,\chi_0)}L(s,\chi)L(s,\bar{\chi}),
\label{e:eq32}
\end{equation}
where $\chi_0$ is the principal character modulo $a$, and this affords an analytic continuation of $f_\chi$ to the whole of $\mathbb C$.

By a quantitative version of Perron's formula, as in Theorem 5.2 of Montgomery \& Vaughan \cite{MV} for example,  we obtain

$$\sideset{}{'}\sum_{n\leq N}a_n(\chi)=\frac{1}{2\pi i}\int_{\sigma_0-iT}^{\sigma_0+iT}f_\chi(s)\frac{N^s}{s}ds+R(\chi),$$
where $\sigma_0>1$ and 
$$ R(\chi)\ll\sum_{\substack{\frac{N}{2}<n<2N\\n\neq N}}|a_n(\chi)|\min\left(1,\frac{N}{T|n-N|}\right)+
\frac{4^{\sigma_0}+N^{\sigma_0}}{T}\sum_{n=1}^{\infty}\frac{|a_n(\chi)|}{n^{\sigma_0}}.$$
Here $\sideset{}{'}\sum$ means that when $N$ is an integer, the term $a_N(\chi)$
is counted with weight $\frac{1}{2}$.

Let $\sigma_0=1+\frac{1}{\log N}$.  By (\ref{e:e31}) we have $|a_n(\chi)|\le 
d(n^2)$.  Thus
$$\sum_{n=1}^{\infty}\frac{|a_n(\chi)|}{n^{\sigma_0}}\ll \zeta(\sigma_0)^3\ll (\log N)^3$$
and so $R(\chi)\ll_\varepsilon N^{1+\varepsilon}T^{-1}$, for any
$\varepsilon>0$.  Hence
$$\sum_{n\leq N}a_n(\chi)=\frac{1}{2\pi i}\int_{\sigma_0-iT}^{\sigma_0+iT}f_\chi(s)\frac{N^s}{s}ds+O\left(\left(\frac{N}{T}+1\right)N^\epsilon\right).$$
The error term here is
$$\ll N^{\varepsilon}$$
provided that
$$T\ge N.$$
The integrand is a meromorphic function in the complex plane and is analytic for all $s$ with $\Re s\ge\frac12$ except for a pole of finite order at $s=1$.  Suppose that $T\ge4$.  By the residue theorem
\begin{align*}
\frac{1}{2\pi i}\int_{\sigma_0-iT}^{\sigma_0+iT}f_\chi(s)\frac{N^s}{s}ds =\,& \text{Res}_{s=1}\left(f_\chi(s)\frac{N^s}{s}\right)+\\
\frac{1}{2\pi i}\left(\int_{\sigma_0-iT}^{\frac{1}{2}-iT}+\int_{\frac{1}{2}-iT}^{\frac{1}{2}+iT}+\int_{\frac{1}{2}+iT}^{\sigma_0+iT}\right)&
\frac{L(s,\chi_0)L(s,\chi)L(s,\bar\chi) N^s}{L(2s,\chi_0)s}ds\\
\end{align*}

We have $L(s,\chi_0)=\zeta(s)\prod_{p|a}(1-p^{-s})$. Hence, by Lemmas 1, 2 and 3 and the fact that 
$\prod_{p|a}(1-p^{-s})\ll\log\log a$ when $\sigma\ge 1$, the contribution from the horizontal paths is 
\begin{align*}
&\ll (\log aT)^2(\log T)(\log\log a))NT^{-1} +T^{-1}(aT)^{\varepsilon}\int_{1/2}^1 (aT)^{\frac{3(1-\sigma)}{2}}N^{\sigma}d\sigma\\
&\ll T^{-1}(aT)^{\varepsilon}N+T^{-1}(aT)^{3/4+\varepsilon}N^{1/2}
\end{align*}
and provided that $a\le 2N$ and $T\ge N^{10}$ this is 
$$\ll N^{-1}.$$
\par
On the other hand, by Lemma 1 the contribution from the vertical path on the right is bounded by
$$N^{\frac12}\left(
\prod_{p|a}(1-p^{-\frac12})^{-1}
\right)(\log T)\sum_{2^k\le T}2^{-k}I(k,\chi)$$
where
$$I(k,\chi)=\int_{-2^{k+1}}^{2^{k+1}}\textstyle|\zeta(\frac12+it)L(\frac12+it,\chi)L(\frac12+it,\Bar\chi)|dt.$$ 
By Lemmas 4 and 5 and H\"older's inequality
\begin{align*}\sum_{\substack{\chi\\ \mod a}} \frac{1}{2\pi i} & \int_{\frac{1}{2}-iT}^{\frac{1}{2}+iT} \frac{L(s,\chi_0)L(s,\chi)L(s,\bar\chi) N^s}{L(2s,\chi_0)s}ds\\
& \ll N^{\frac12}\left(
\prod_{p|a}(1-p^{-\frac12})^{-1}
\right)(\log T)\sum_{2^k\le T} a (k+\log a)^3\\
& \ll N^{\frac12}\left(
\prod_{p|a}(1-p^{-\frac12})^{-1}
\right)a(\log N)^5
\end{align*}
on taking 
$$T=N^{10}.$$
Thus we have shown that
$$S(N;a)=\frac{1}{\phi(a)}\sum_{\substack{\chi\\ \mod a}}\bar\chi(-1)\text{Res}_{s=1}\left(f_\chi(s)\frac{N^s}{s}\right) + \Delta(N;a)$$
where
$$\Delta(N;a)\ll N^{\frac12}(\log N)^5\frac{a}{\phi(a)}\prod_{p|a}\left(
1-p^{-1/2}
\right)^{-1}$$
It remains to compute the residue at $s=1$.\par

By (\ref{e:eq32}) there are naturally two cases, namely, $\chi\neq\chi_0$ and $\chi= \chi_0$.  When $\chi\neq \chi_0$  the integrand has a simple pole at $s=1$ and the residue is
$$\text{Res}_{s=1}\left(
\frac{L(s,\chi_0)L(s,\chi)L(s,\bar\chi) N^s}{L(2s,\chi_0)s}
\right)=\frac{6N}{\pi^2}\left(
\prod_{p|a}\frac{p}{p+1}
\right)|L(1,\chi)|^2.$$
It is useful to have some understanding of the behavior of
$$\frac1{\phi(a)}\sum_{\substack{\chi\not=\chi_0\\ \mod a}} \Bar\chi(-1)|L(1,\chi)|^2.$$
Let $x=a^3$.  Then for non-principal characters $\chi$ modulo $a$, by Abel summation 
$$L(1,\chi)=\sum_{n\le x}\frac{\chi(n)}n + O(a^{-2}).$$ 
Hence
\begin{align*}
\frac1{\phi(a)}\sum_{\substack{\chi\not=\chi_0\\ \mod a}}& \Bar\chi(-1)|L(1,\chi)|^2 \\
&= \frac1{\phi(a)}\sum_{\substack{\chi\not=\chi_0\\ \mod a}} \Bar\chi(-1)\left|
\sum_{n\le x}\frac{\chi(n)}n
\right|^2 +O\big(a^{-1}\big).
\end{align*}
The main term on the right is
$$\frac1{\phi(a)}\sum_{\substack{\chi\\ \mod a}} \Bar\chi(-1)\left|
\sum_{n\le x}\frac{\chi(n)}n
\right|^2 - \frac1{\phi(a)}\Bigg(\sum_{\substack{n\le x\\(n,a)=1}}\frac1n
\Bigg)^2.$$
We have
\begin{align*}
\sum_{\substack{n\le x\\(n,a)=1}}\frac1n&= \sum_{m|a}\frac{\mu(m)}m\sum_{n\le x/m}\frac1n\\
&=\sum_{m|a}\frac{\mu(m)}m(\log(x/m)+\gamma+O(m/x))\\
&=\frac{\phi(a)}a\Bigg(\log x+\gamma+\sum_{p|a}\frac{\log p}{p-1}\Bigg) +O(d(a)/x).
\end{align*}
Hence the second term above is
$$-\frac{\phi(a)}{a^2}\Bigg(\log x+\gamma+\sum_{p|a}\frac{\log p}{p-1}\Bigg)^2 +O(1/a).$$
The first term above is
$$\sum_{\substack{m,n\le x\\(mn,a)=1\\a|m+n}}\frac1{mn}.$$
The terms with $m=n$ contribute
$$\sum_{\substack{m\le x\\(m,a)=1\\a|2m}}\frac1{m^2}\ll a^{-2}$$
and this can be collected in the error term.  The remaining terms are collected together so that $m+n=ak$, $m\not=n$ and $k\le \frac{2x}{a}$.  If necessary by interchanging $m$ and $n$ we can suppose that $m<n$.  Thus the above is
$$\sum_{1\le k\le 2x/a}\sum_{\substack{m\le x\\ 0<ak-m\le x\\ m\le ak/2\\(m,a)=1}}\frac2{m(ak-m)}.$$
On interchanging the order of summation this becomes
$$\sum_{\substack{m\le x\\(m,a)=1}} \frac2m \sum_{\substack{2m/a<k\le(x+m)/a}}\frac1{ak-m}.$$
We now divide the sum over $m$ according as $m>a/2$ or $m\le a/2$.  In the former case
the inner sum can be written as the Stieltjes integral
$$\int_{(2m/a)+}^{(x+m)/a+} \frac{d\lfloor\alpha\rfloor}{a\alpha-m}= \frac{\lfloor(x+m)/a\rfloor}{x} - \frac{\lfloor 2m/a\rfloor}m + \int_{2m/a}^{(x+m)/a} \frac{a\lfloor\alpha\rfloor}{(a\alpha-m)^2}d\alpha.$$
Since $m\le x$ the first term is $\ll 1/a$, and the second term is 0 unless $m\ge \frac{a}2$, in which case it is $\ll 1/a$. Thus these terms contribute $\ll (\log a)/a$ in total.  The integral here is
$$\int_{2m/a}^{(x+m)/a} \frac{a\alpha-m -a(\alpha-\lfloor\alpha\rfloor)+m}{(a\alpha-m)^2}d\alpha=a^{-1}\log(x/m) + O(1/a).$$  
Thus the contribution to our sum is
$$a^{-1}\sum_{\substack{a/2<m\le x\\(m,a)=1}} \frac2m\log(x/m) + O\big((\log a)a^{-1}\big)$$
When $m\le a/2$ the sum over $k$ becomes instead
$$\int_{1-}^{(x+m)/a+} \frac{d\lfloor\alpha\rfloor}{a\alpha-m}= \frac{\lfloor(x+m)/a\rfloor}{x} + \int_1^{(x+m)/a} \frac{a\lfloor\alpha\rfloor}{(a\alpha-m)^2}d\alpha.$$
The first term is $\ll 1/a$ and the integral is
$$\int_1^{(x+m)/a} \frac{a\alpha-m -a(\alpha-\lfloor\alpha\rfloor)+m}{(a\alpha-m)^2}d\alpha=a^{-1}\log(x/(a-m)) + O(1/a).$$
Thus we have shown that
\begin{align*}\frac1{\phi(a)}&\sum_{\substack{\chi\not=\chi_0\\ \mod a}} \Bar\chi(-1)|L(1,\chi)|^2 = a^{-1}\sum_{\substack{m\le x\\(m,a)=1}} {\textstyle\frac2m}\log{\textstyle\frac{x}{m}} - a^{-1}\sum_{\substack{m\le a/2\\(m,a)=1}} {\textstyle\frac2m}\log{\textstyle\frac{a-m}{m}}\\
& -\frac{\phi(a)}{a^2}\Bigg(\log x+\gamma+\sum_{p|a}\frac{\log p}{p-1}\Bigg)^2+ O((\log a)a^{-1})\end{align*}
The first sum on the right is
$$2a^{-1}\sum_{k|a}\frac{\mu(k)}{k} \sum_{n\le x/k} n^{-1}\log(x/kn)$$
and this is readily seen to be
$$a^{-1}\sum_{k|a}\frac{\mu(k)}{k}\Big((\log(x/k))^2+2\gamma\log(x/k)+C\Big) +O(d(a)/(ax))$$
for a suitable constant $C$.  Here the main term is
$$\frac{\phi(a)}{a^2}\Bigg(
\Big(\log x+\gamma+\sum_{p|a}\frac{\log p}{p-1}\Big)^2 +O\big((\log\log(3a))^2\big)
\Bigg).$$
Hence, we have
\begin{align*}
\frac1{\phi(a)}\sum_{\substack{\chi\not=\chi_0\\ \mod a}}& \Bar\chi(-1)|L(1,\chi)|^2\\
&= - a^{-1}\sum_{\substack{m\le a/2\\(m,a)=1}} {\textstyle\frac2m}\log{\textstyle\frac{a-m}{m}} +O( a^{-1}\log 2a).
\end{align*}
The sum over $m$ is
\begin{align*}
\sum_{\substack{m\le a/2\\(m,a)=1}} &{\textstyle\frac2m}\log{\textstyle\frac{a/2}{m}} +O( \log 2a)\\
& = \frac{\phi(a)}{a}\Bigg(
(\log(a/2))^2-2(\log(a/2)\sum_{p|a}\frac{\log p}{p-1}\Bigg) +O(\log 2a)
\end{align*}
\par

When $\chi=\chi_0$, we have
\begin{align*}
f_{\chi}(s)&=\frac{L^3(s,\chi_0)}{L(2s,\chi_0)}\\
    &=\frac{\zeta^3(s)\prod_{p|a}\left(1-\frac{1}{p^s}\right)^3}{\zeta(2s)\prod_{p|a}\left(1-\frac{1}{p^{2s}}\right)}\\
    &=\frac{\zeta^3(s)}{\zeta(2s)}\prod_{p|a}\frac{(p^s-1)^2}{p^s(p^s+1)}.
\end{align*}
Let
$$F(s)=\big((s-1)\zeta(s)\big)^3\zeta(2s)^{-1}s^{-1},$$
$$G(s)=\prod_{p|a}\frac{(p^s-1)^2}{p^s(p^s+1)}$$
and 
$$H(s)=F(s)G(s).$$
Then $H$ has a removable singularity at $s=1$ and we are concerned with the residue of 
$$(s-1)^{-3}N^sH(s)$$
at $s=1$.  This is
$$\frac12 N(\log N)^2H(1)+ N(\log N)H'(1) +\frac12 NH''(1)$$
which it is convenient to rewrite as
$$NH(1)\left(
\textstyle \frac12(\log N)^2 + (\log N)\frac{H'(1)}{H(1)} + \frac{H''(1)}{2H(1)}
\right).$$
Now 
$$\frac{H'(1)}{H(1)}=\frac{F'(1)}{F(1)}+\frac{G'(1)}{G(1)}$$
and
$$\frac{H''(1)}{H(1)}=\frac{F''(1)}{F(1)}+2\frac{F'(1)G'(1)}{F(1)G(1)}+\frac{G''(1)}{G(1)}$$
and $F'(1)/F(1)$ and $F''(1)/F(1)$ can be evaluated in terms of Euler's and Stieltje's constants and $\zeta(2)$ and its derivatives.  In particular
$$\textstyle \frac{F'(1)}{F(1)} = 3\gamma -2\frac{\zeta'(2)}{\zeta(2)}-1.$$
The function $G$ is more interesting.  We have
$$\textstyle \frac{G'(1)}{G(1)} = \sum_{p|a} \frac{3p+1}{(p-1)^2}\log p$$
and
$$\textstyle \frac{G''(1)}{G(1)}=\left(
\frac{G'(1)}{G(1)}
\right)^2 - \sum_{p|a}\frac{3p^3+2p^2+3p}{(p^2-1)^2}(\log p)^2.$$
Thus 
$$\textstyle \frac{G'(1)}{G(1)} \ll \log\log(3a)$$
and
$$\textstyle \frac{G''(1)}{G(1)} \ll (\log\log(3a))^2$$

\hfill $\square$

\section{Proof of Theorem \ref{t:thm2}}
By the same argument in the beginning of section 3, $R(n;a)$ can be reduced to counting the number of divisors $u$ of $n^2$ with $u+n\equiv 0\pmod{a}$. 
Now the condition $u|n^2$ can be rewritten uniquely as $u=n_1n_{2}^2$ and $n=n_1n_2n_3$ with $n_1$ being square-free. Thus
we have 
\begin{eqnarray*}
R(n;a)&=&\sum_{\substack{u|n^2\\a|u+n}}1\\
&=&\sum_{\substack{n_1n_2n_3=n\\a|n_2+n_3}}\mu^2(n_1)
\end{eqnarray*}
and hence 
$$U(N)=\sum_{n_1\le N}\mu^2(n_1)\left(\sum_{n_2n_3\le N/n_1}\sum_{\substack{a|n_2+n_3\\(a,n_1n_2n_3)=1}}1\right).$$

The inner double sum is symmetric in $n_2$ and $n_3$, so writing $M=N/n_1$ and using Dirichlet's method of the hyperbola it is
$$\sum_{n_2\le \sqrt{M}}\sum_{\substack{a\le n_2+M/n_2\\(a,n_1n_2)=1}}\sum_{\substack{n_3\le M/n_2\\n_3\equiv -n_2\pmod{a}}}2
-\sum_{n_2\le \sqrt{M}}\sum_{\substack{a\le n_2+\sqrt{M}\\(a,n_1n_2)=1}}\sum_{\substack{n_3\le \sqrt{M}\\n_3\equiv -n_2\pmod{a}}}1.$$
The second triple sum here is $\ll \sum_{n_2\le \sqrt{M}}\sum_{a\le n_2+\sqrt{M}}\frac{\sqrt{M}}{a}\ll M\log M$, leading to a contribution 
$\ll N(\log N)^2$ in the original sum. The first triple sum is 
$$\sum_{n_2\le\sqrt{M}}\sum_{\substack{a\le n_2+M/n_2\\(a,n_1n_2)=1}}\frac{2M}{a n_2}$$
with an error $\ll M\log M$. The $a$ in the range $(M/n_2,n_2+M/n_2]$ are of order of magnitude $M/n_2$ and there are at most $n_2$ of them, so the total contribution from this part of the sum is $\ll M$, and the contribution from this to the original sum is $\ll N\log N$. Thus we are left with
$$\sum_{n_2\le \sqrt{M}}\sum_{\substack{a\le M/n_2\\(a,n_1n_2)=1}}\frac{2M}{a n_2}.$$  

Now using the M\"obius function to pick out the condition $(a,n_1n_2)=1$, the inner sum over $a$ can be written as
$$\sum_{k|n_1n_2}\frac{\mu(k)}{k}\sum_{b\le M/(n_2k)}\frac{2M}{b n_2}.$$

Put $k_1=(k,n_1)$, $k_2=k/k_1$, $n_1'=n_1/k_1$, so that $k_2|n_2$, $(k_2,n_1')=1$, and let $n_2'=n_2/k_2$. Observe also that for $\mu(n_1)=\mu(n_1'k_1)=\mu(n_1')\mu(k_1)\not=0$ it is necessary that $(n_1',k_1)=1$.  Thus substituting in the original sum gives
$$\sum_{k_1\le N}\sum_{k_2\le N}\frac{\mu(k_1k_2)}{k_1^2k_2^2}\sum_{\substack{n_1'\le N/k_1\\(n_1',k_1k_2)=1}}\frac{\mu^2(n_1')}{n_1'}
\sum_{n_2'\le k_2^{-1}\sqrt{N/(n_1'k_1)}}\,\sum_{b\le N/(n_1'n_2'k_1^2k_2^2)}\frac{2N}{b n_2'}$$
and there are various implications for a non-zero contribution. Thus 
$$n_1'n_2'k_1^2k_2^2\le N$$ 
and this is a more stringent condition on $n_2'$ than $n_2'\le k_2^{-1}\sqrt{N/(n_1'k_1)}$ when $n_2'\le k_1$. Also $n_1'\le N/(k_1^2k_2^2)$ and $k_1k_2\le \sqrt{N}$. The sum over $b$ is
 $$\log\left(N/(n_1'n_2'k_1^2k_2^2)\right)+O(1).$$
Consider the error term here. The sum over $n_1'$ and $n_2'$ contributes
$$\ll N(\log N)^2.$$
Thus one is left to consider
$$\sum_{\substack{k_1,k2\\
k_1k_2\le \sqrt{N}}}\frac{\mu(k_1k_2)}{k_1^2k_2^2}\sum_{\substack{n_1'\le N/(k_1^2k_2^2)\\(n_1',k_2)=1}}\frac{\mu^2(n_1')}{n_1'}
\sum_{\substack{n_2'\le k_2^{-1}\sqrt{N/(n_1'k_1)}\\n_2'\le N/(n_1'k_1^2k_2^2)}}\frac{2N}{n_2'}\log\left(\frac{N}{n_1'n_2'k_1^2k_2^2}\right).$$
The $n_2'$ with $n_2'^2k_2^2n_1'k_1\le N< n_2'n_1'k_1^3k_2^2$ satisfy $n_2'\le k_1$ so they would contribute $\ll N(\log k_1)\log N$ to the innermost sum and hence give a total contribution of $\ll N(\log N)^2$.
Thus we can ignore the condition $n_2'\le N/(n_1'k_1^2k_2^2)$.

Now the the summation over $n_2'$ can be performed and this gives
$$2N\left(\frac12L_1^2+L_1L_2\right)$$
where
$$L_1=\log\frac{\sqrt{N}}{k_2\sqrt{n_1'k_1}}$$ 
and
$$L_2=\log\frac{\sqrt{N}}{k_1k_2\sqrt{n_1'k_1}}$$
with an error $\ll N\log N$ and a total error $\ll N(\log N)^2$. Now let 
$$L=\log\frac{N}{k_1^2k_2^2n_1'},$$ 
then the above expression is easily seen to be a quadratic polynomial in $L$, \emph{i.e.}
\begin{eqnarray*}
&&2N\left(\frac12L_1^2+L_1L_2\right)\\
&=&\frac{1}{2}\left(\frac12(L+\log k_1)^2+(L+\log k_1)(L-\log k_1)\right)\\
&=&\frac14\Big(3L^2+2(\log k_1)L-(\log k_1)^2\Big).
\end{eqnarray*}

Observe that the major contribution comes from the quadratic term in $L$ here, and the other terms contribute $\ll N(\log N)^2$ in the original sum.
So one is left to deal with
$$\frac34\sum_{k_1\le \sqrt{N}}\sum_{k_2\le \sqrt{N}/k_1}\frac{\mu(k_1k_2)}{k_1^2k_2^2}\sum_{\substack{n_1'\le N/(k_1^2k_2^2)\\(n_1',k_1k_2)=1}}\frac{\mu^2(n_1')}{n_1'}N\left(\log\frac{N}{k_1^2k_2^2n_1'}\right)^2$$
$$=\frac34 \sum_{k\le \sqrt N} \frac{\mu(k)d(k)}{k^2} \sum_{\substack{n\le N/k^2\\(n,k)=1}}\frac{\mu^2(n)}{n}N\left(\log\frac{N}{k^2n}\right)^2.$$
When $\theta>0$ it follows by absolute convergence that the above sum is
$$\frac3{4\pi i} \int_{\theta-i\infty}^{\theta+i\infty} \zeta(1+s)D(1+s) \frac{N^{s+1}}{s^3} ds$$
where
$$D(s)=\prod_p\left(
1-\frac{3}{p^{2s}}+\frac{2}{p^{3s}}
\right).$$
The Euler product $D(s)$ converges locally uniformly for $\Re s>\frac12+\delta$ for any $\delta>0$.  Hence, by standard estimates for the Riemann zeta function the vertical path may be moved to the vertical path $\Re s=\psi$ where $-\frac12<\psi<0$, picking up the residue of the pole of order $4$ at $s=0$.  It follows that
\begin{align*}
\frac34 \sum_{k\le \sqrt N} \frac{\mu(k)d(k)}{k^2} \sum_{\substack{n\le N/k^2\\(n,k)=1}}\frac{\mu^2(n)}{n}&N\left(\log\frac{N}{k^2n}\right)^2\\
& = \frac32 N \frac{(\log N)^3}{6} D(1) +O(N\log^2 N).
\end{align*}
This establishes the theorem.

\hfill $\square$

\section{Further Comments}
The corresponding questions for the equation (\ref{e:eq1}) when $k\ge 3$ are still open.  Indeed, whilst it follows from the criterion in the second paragraph of \S3 that
$$R(n;a)\ll n^\varepsilon,$$
and generally one could conjecture that $R_k(n;a)$, the number of solutions of (\ref{e:eq1}) in positive integers, satisfies 
the concomitant bound
$$R_k(n;a)\ll n^\varepsilon,$$
this is far from what has been established.  Indeed, if we define $S_k(N;a)$ for general $k$ by
$$S_k(N;a)=\sum_{\substack{n\leq N\\(n,a)=1}}R_k(n;a)$$  
when $k\ge3$ it has not even been established that
$$S_k(N;a)\ll N^{1+\varepsilon}.$$
\par
It seems likely that
$$S_k(N;a)\sim CN(\log N)^{\alpha},$$
for some positive constants $C$ and $\alpha$ which only depend on $k$ and, in the case of $C$, on $a$.   One can also make similar conjectures for the corresponding $T_k(N;a)$ and $U_k(N)$.


\begin{thebibliography}{xx}

\bibitem{Cro} C. Croot; D. Dobbs; J. Friedlander; A. Hetzel; F. Pappalardi, \emph{Binary Egyptian fractions}, J. Number Theory 84(2000), no. 1, 63-79.

\bibitem{Guy}R. K. Guy, \emph{Unsolved problems in Number Theory}, second edition. Springer-Verlag, 1994.

\bibitem{Ing}A. E. Ingham, \emph{Mean-value theorems in the theory of the Riemann zeta function},
Proc. London Math. Soc. (2) 27 (1926) 273-300.


\bibitem{Mon}H.L. Montgomery, \emph{Topics in Multiplicative Number Theory}, Springer-Verlag, 1971.

\bibitem{MV}H.L. Montgomery and R.C. Vaughan, \emph{Multiplicative Number Theory I. Classical Theory}, Cambridge University Press, 2007.

\bibitem{Shen} Shen Zun, \emph{On the diophantine equation $\sum_{i=0}^{k}\frac{1}{x_i}=\frac{a}{n}$}, Chinese Annals of Mathematics, 1986, 7B: 213-220.

\bibitem{Tit} E.C. Titchmarsh, \emph{The Riemann Zeta-Function, 2nd edition, revised by D.R. Heath-Brown}, Oxford, 1986.
\bibitem{Vau}R.C. Vaughan, \emph{An application of the large sieve to a diophantine equation}, Berichte aus dem Mathematischen Forschungsinstitut Oberwolfach, Heft 5, Bibliographisches Institut, March 1970: 203-207.
\bibitem{Vau1}R.C. Vaughan, \emph{On a problem of Erd\"{o}s, Straus and Schinzel}, Mathematika, 1970, 17: 193-198.

\bibitem{Vio} C. Viola, \emph{On the diophantine equation $\prod_{i=0}^{k}{x_i}-\sum_{i=0}^{k}{x_i}=n$ and $\sum_{i=0}^{k}\frac{1}{x_i}=\frac{a}{n}$}, Acta Arithmetica, 1973, 22: 339-352.

\end{thebibliography}
\end{document}